\documentclass[a4paper]{amsart}

\usepackage[utf8]{inputenc}
\usepackage{mathtools}
\usepackage{amssymb,amsfonts,amsmath}
\usepackage{enumitem}
\usepackage{color}
\usepackage{mathrsfs}
\usepackage{graphicx}
\usepackage{verbatim}
\usepackage{tikz-cd}
\usepackage{tikz}

\numberwithin{equation}{section}
\newtheorem{theorem}{Theorem}[section]
\newtheorem{corollary}[theorem]{Corollary}
\newtheorem{lemma}[theorem]{Lemma}
\newtheorem{proposition}[theorem]{Proposition}

\newtheorem{problem}[theorem]{Problem}

\newtheorem*{theorem*}{Theorem}

\theoremstyle{definition}
\newtheorem{definition}[theorem]{Definition}
\newtheorem{notation}[theorem]{Notation}

\theoremstyle{remark}

\newcommand{\F}{\mathbb{F}}
\newcommand{\N}{\mathbb{N}}

\DeclareMathOperator{\ab}{ab}

\DeclareMathOperator{\Alt}{Alt}

\DeclareMathOperator{\ord}{o}

\DeclareMathOperator{\pr}{pr}

\DeclareMathOperator{\Sym}{Sym}
\DeclareMathOperator{\T}{T}
\DeclareMathOperator{\Tr}{Tr}

\newcommand{\abs}[1]{\vert #1 \vert}
\newcommand\Set[2]{\{\,#1\mid#2\,\}}

\newcommand{\defeq}{\mathrel{\mathop{:}}=}

\renewcommand{\epsilon}{\varepsilon}

\title[Finitely generated infinite torsion groups that are residually finite simple]{Finitely generated infinite torsion groups that are residually finite simple}
\author[E. Schesler]{Eduard Schesler}
\address{Erwin Schrödinger International Institute for Mathematics and Physics \\ University of Vienna \\ 1090 Vienna}

\email{eduardschesler@googlemail.com}

\begin{document}
\begin{abstract}
We show that every finitely generated residually finite torsion group $G$ embeds in a finitely generated torsion group $\Gamma$ that is residually
finite simple.
In particular we show the existence of finitely generated infinite torsion groups that are residually
finite simple,
which answers a question of Olshanskii and Osin.
\end{abstract}
\maketitle

\section{Introduction}

Let $\mathcal{C}$ be a class of groups.
A group $G$ is said to be \emph{residually $\mathcal{C}$} if the intersection of all normal subgroup $N$ of $G$ with $G / N \in \mathcal{C}$ is the trivial group.
It is a classical problem in group theory to determine the classes of groups $\mathcal{C}$ for which a given group is residually $\mathcal{C}$ and a lot of research has been done in this direction, see e.g.~\cite{Weigel92,Wise02,Agol08,BridsonEvansLiebeckSegal19}.
A special instant of this problem was formulated in 1987 by Gromov~\cite{Gromov87} and became a notorious open problems in geometric group theory: Is every hyperbolic group residually finite, i.e.\ residually $\mathcal{F}$, where $\mathcal{F}$ denotes the class of finite groups.
In 2008 it was shown by Olshanskii and Osin~\cite{OlshanskiiOsin08} that an affirmative answer to Gromov's question would imply the existence of finitely generated infinite torsion groups that are residually $\mathcal{FS}$, where $\mathcal{FS}$ denotes the class  of finite simple groups.
It was therefore natural for them to ask the following, see~\cite[Problem 3.4]{OlshanskiiOsin08}.

\begin{problem}\label{prob:intro}
Does there exist an infinite finitely generated torsion group that is residually $\mathcal{FS}$?
\end{problem}

Despite of a variety of techniques that are known to produce infinite finitely generated residually finite torsion groups that range from amenable~\cite{Grigorchuk80,GuptaSidki83}  and non-amenable branch groups~\cite{SidkiWilson03,KionkeScheslerAmenableProfinite} to groups with property $(\T)$~\cite{Ershov08,ErshovJaikin13}, and groups with positive first $\ell^2$-Betti number~\cite{LueckOsin11,KionkeSchesler24}, there was no construction known so far that produces infinite finitely generated torsion groups that are residually $\mathcal{FS}$.
In fact there is a big obstruction for infinite finitely generated torsion groups to be residually $\mathcal{FS}$.
To make this more precise, let us write $\mathcal{FS}_{k}$ to denote the subclass of $\mathcal{FS}$ that consists groups that do not contain a subgroup isomorphic to $\Alt(k)$.
It was pointed out by Lubotzky and Segal~\cite[Theorem 16.4.2(i)]{LubotzkySegal03} that every finitely generated group $G$ that is residually $\mathcal{FS}_{k}$ for some $k$ can be realized as a subdirect product of finitely many linear groups.
In particular, if such a group $G$ is infinite, it admits an infinite finitely generated linear quotient, which is virtually torsion free.
It therefore follows that the class of finitely generated groups that are residually $\mathcal{FS}_k$ for some $k$ does not contain an infinite torsion group.
In view of this, it can be easily seen that an affirmative answer to Problem~\ref{prob:intro} implies the existence of a torsion group $\Gamma$ that is a subdirect subgroup of a product $\prod_{i=1}^{\infty} S_i$, where $S_i \in \mathcal{FS}$ contains an isomorphic copy of $\Alt(i)$.
We will show that such a group $\Gamma$ indeed exists and thereby answer Problem~\ref{prob:intro} affirmatively.
In fact we will see that every finitely generated residually finite torsion group embeds in a group $\Gamma$ as above.

\begin{theorem}\label{thm:main-intro}
Every finitely generated residually finite torsion group embeds into a torsion group that is residually $\mathcal{FS}$.
\end{theorem}

The proof of Theorem~\ref{thm:main-intro} is based on the following idea.
Consider a group $G$ and a sequence of $G$-sets $(\Omega_i)_{i \in \N}$ that are represented by homomorphisms $\alpha_i \colon G \rightarrow \Sym(\Omega_i)$.
For each $i$ let $\tau_i \in \Sym(\Omega_i)$ be a permutation of $\Omega_i$.
Then, under suitable assumptions on $\tau_i$ and $\alpha_i$, the subgroup $\Gamma$ of $\prod_{i \in \N} \Sym(\Omega_i)$ that is generated by $(\tau_i)_{i \in \N}$ and the image of
\[
\alpha \colon G \rightarrow \prod_{i \in \N} \Sym(\Omega_i),\ g \mapsto (\alpha_i(g))_{i \in \N}
\]
will keep some of the properties of $G$, e.g.\ being torsion, while gaining some extra properties, e.g.\ being residually $\mathcal{FS}$.
A related idea was recently applied in a work of Kionke and the author~\cite{KionkeScheslerTelescopes} in order to produce new examples of infinite finitely generated amenable simple groups.





\subsection*{Acknowledgments}
This article arose from the author's research stays at the Erwin Schrödinger International Institute for Mathematics and Physics in Vienna and the Institute of Mathematical Sciences in Madrid.
The author would like to thank these institutes for their financial and organizational support.
The author is grateful to Goulnara Arzhantseva, Andrei Jaikin-Zapirain, Steffen Kionke, and Markus Steenbock for helpful discussions.

\section{Extending actions of torsion groups}

For the rest of this section we fix a torsion group $G$ that acts on a set $\Omega$.
Let $\alpha \colon G \rightarrow \Sym(\Omega)$ denote the corresponding homomorphism.
Let us moreover fix an element $p \in \Omega$ and let $\Omega^{+} \defeq \Omega \cup \{q\}$ for some $q \notin \Omega$.
We are interested in the subgroup $\Gamma$ of $\Sym(\Omega^{+})$ that is generated by $\alpha(G)$ and the transposition $\tau = (p,q)$.


\begin{notation}\label{not:initial-subword}
Let $F(X)$ denote the free group over a set $X$ and let $w = x_{i_1} \ldots x_{i_{\ell}} \in F(X)$ be a reduced word of length $\ell \in \N_0$.
For each $0 \leq k \leq \ell$ we write
$w_{\geq k} \defeq x_{i_{\ell-k+1}} \ldots x_{i_{\ell}}$ to denote the terminal subword of length $k$ in $w$.
\end{notation}

Let us now consider the free group $F \defeq F(G \cup \{\tau\})$.
To simplify the notation we will often interpret a word $w \in F$ as an element of $\Gamma$, respectively $G$ if $w \in F(G)$, as long as no ambiguity is possible.

\begin{definition}\label{def:w-trace}
For each word $w \in F$ of length $\ell \in \N_0$ and each point $\xi \in \Omega^{+}$, we define the \emph{$w$-trace of $\xi$}
as the sequence
\[
\Tr_w(\xi) \defeq (w_{\geq i} \cdot \xi)_{i=1}^{\ell}.
\]
\end{definition}

Note that the $w$-trace of an element $\xi$ does not necessarily contain $\xi$.
Let us now fix a finite sequence $g_1,\ldots,g_k$ of elements in $G$.
In what follows we will study traces for the words
\[
v_{n,i} = (g_1 \ldots g_k)^n g_1 \ldots g_i
\]
and
\[
w_{n,i} = (\tau g_1 \ldots \tau g_k)^n \tau g_1 \ldots \tau g_i
\]
in $F$, where $n \in \N_0$ and $0 \leq i < k$.


\begin{notation}\label{not:order}
Given a group $H$ and an element $h \in H$, we write $\ord_H(h) \in \N \cup \{\infty\}$ to denote the order of $h$ in $H$.
\end{notation}

Let us consider the element $g \defeq g_1 \ldots g_k \in G$ and let $N = \ord_G(g)$.


\begin{lemma}\label{lem:fundamental-xi}
Let $\xi \in \Omega^{+}$ and let $0 \leq i < k$.
Suppose that $p$ is not contained in $\Tr_{w_{N,i}}(\xi)$.
Then $p$ is not contained in $\Tr_{w_{n,i}}(\xi)$ for every $n \in \N_0$.
\end{lemma}
\begin{proof}
If $\xi = q$, then $p$ is clearly contained in $\Tr_{w_{N,i}}(\xi)$ so that there is nothing to show.
Let us therefore assume that $\xi \in \Omega$ and that $\Tr_{w_{N,i}}(\xi)$ does not contain $p$.
Since $\tau$ fixes every point in $\Omega \setminus \{p\}$, it follows that $\Tr_{v_{N,i}}(\xi)$ does not contain $p$.
Thus there is no non-trivial terminal subword $u$ of $(g_1 \ldots g_k)^N g_1 \ldots g_i$ that satisfies $u(\xi) = p$.
Since
\[
(g_1 \ldots g_k)^N g_1 \ldots g_i \cdot \xi
= g_1 \ldots g_i \cdot \xi,
\]
it follows that $\Tr_{v_{aN+r,i}}(\xi)$ does not contain $p$ for every $a \in \N_0$ and every $r < k$.
Thus the same is true for $w_{aN+r,i}$, which proves the lemma.
\end{proof}

\begin{lemma}\label{lem:fundamental-p-0}
The element $p$ is contained in $\Tr_{w_{N,i}}(p)$ for every $0 \leq i < k$.
\end{lemma}
\begin{proof}
Suppose that $p$ is not contained in $\Tr_{w_{N,i}}(p)$.
Since $\tau$ fixes every point in $\Omega \setminus \{p\}$, it follows that $p$ is not contained in $\Tr_{v_{N,i}}(p)$.
However, this is not possible since the word $(g_{i+1} \ldots g_k g_1 \ldots g_i)^N$, which represents the trivial element in $G$, is a non-trivial terminal subword of $v_{N,i}$.
\end{proof}

\begin{lemma}\label{lem:fundamental-p-1}
Let $\xi \in \Omega^{+}$, let $n \in \N_0$, and let $0 \leq i < k$.
Suppose that $\Tr_{w_{n,i}}(\xi)$ contains $p$.
Then there are natural numbers $m_1,m_2,j$ with $0 \leq m_1 < m_2 < N(k+1)$ and $0 \leq j < k$ such that
\[
w_{N(k+1),0} \cdot \xi
= w_{m_1,j} \cdot p
= w_{m_2,j} \cdot p.
\]
\end{lemma}
\begin{proof}
From Lemma~\ref{lem:fundamental-xi}
we know that $\Tr_{w_{N,0}}(\xi)$ contains $p$.
Thus there are integers $n_1 < N$ and $i_1 < k$ with
\[
w_{N,0} \cdot \xi
= w_{n_1,i_1} \cdot p
\]
and therefore
\[
w_{(k+1)N,0} \cdot \xi
= w_{kN+n_1,i_1} \cdot p.
\]
Now an inductive application of Lemma~\ref{lem:fundamental-p-0} provides us with integers $n_2,n_3,\ldots,n_k < N$ and $i_2,i_3,\ldots,i_k < k$ such that
\begin{align*}
w_{kN+n_1,i_1} \cdot p
&= w_{(k-1)N+n_1+n_2,i_2} \cdot p\\
&= w_{(k-2)N+n_1+n_2+n_3,i_3} \cdot p\\
&\ \vdots\\
&= w_{n_1+ \ldots + n_{k+1},i_{k+1}} \cdot p.
\end{align*}
Regarding this, the lemma follows from the pigeonhole principle applied to the sequence of indices $i_1,\ldots,i_{k+1}$.
\end{proof}

\begin{lemma}\label{lem:fundamental-general}
For every $\xi \in \Omega^{+}$ there is a natural number $m \leq N(k+1)$ such that $w_{m,0}(\xi) = \xi$.
\end{lemma}
\begin{proof}
Suppose first that $p$ is not contained in $\Tr_{w_{N,0}}(\xi)$.
Then $p$ is not contained in $\Tr_{v_{N,0}}(\xi)$ and we obtain
\[
w_{N,0} \cdot \xi
= (g_1 \ldots g_k)^N \cdot \xi
= \xi.
\]
Suppose next that $p$ is contained in $\Tr_{w_{N,0}}(\xi)$.
From Lemma~\ref{lem:fundamental-p-1} we know that there are natural numbers $m_1,m_2,j$ with $0 \leq m_1 < m_2 < N(k+1)$ and $0 \leq j < k$ such that
\[
w_{N(k+1),0} \cdot \xi
= w_{m_1,j} \cdot p
= w_{m_2,j} \cdot p.
\]
In view of this, we see that
$w_{m_2-m_1,0} \cdot \xi = \xi$, where $m_2 - m_1 \leq (k+1)N$.
\end{proof}


\section{Embedding torsion groups}

In this section we will apply Lemma~\ref{lem:fundamental-general} in the case where the involved groups are finitely generated and residually finite.
This will enable us to prove Theorem~\ref{thm:main-intro} from the introduction.

\subsection{The finitely generated case}


Let $G$, $\Gamma$, and $\Omega^{+}$ be as above.
Suppose now that $G$ is finitely generated and let $X$ be a finite generating set of $G$.
In this case we can define the \emph{torsion growth function of $G$ with respect to $X$} as the function
\[
T_G^X \colon \N \rightarrow \N,\ \ell \mapsto \max \Set{\ord_G(g)}{g \in B_G^X(\ell)},
\]
where $B_G^X(\ell)$ denotes the set of elements of $G$ whose word length with respect to $X$ is bounded above by $\ell$.
We consider the generating set $X^{+} \defeq \alpha(X) \cup \{\tau\}$ of $\Gamma$.

\begin{lemma}\label{lem:fundamental-orbit}
Let $\ell \in \N$, let $\gamma \in B_{\Gamma}^{X^{+}}(\ell)$, and let $\xi \in \Omega^{+}$.
The size of the orbit $\langle \gamma \rangle \cdot \xi$ is bounded above by $T_G^{X}(\ell) \cdot (\ell+1)$.
\end{lemma}
\begin{proof}
Since the claim is trivial otherwise, we may assume that $\gamma$ does not lie in $B^{\alpha(X)}_{\alpha(G)}(\ell)$.
Thus, up to conjugation,
we may assume that $\gamma$ is represented by a word of the form
\[
w = \tau g_1 \tau \ldots \tau g_r,
\]
where $\sum_{i=1}^r \abs{g_i}_{\alpha(X)} \leq \ell$ and therefore $\abs{g_1 \ldots g_r}_X \leq \ell$.
In this case we know from Lemma~\ref{lem:fundamental-general} that there is a natural number
\[
m \leq T_G(\ell)(r+1) \leq T_G(\ell)(\ell+1)
\]
such that $\gamma^m(\xi) = \xi$. 
\end{proof}

Note that Lemma~\ref{lem:fundamental-orbit} has the following immediate consequence.

\begin{corollary}\label{cor:fundamental-torsion}
Every element $\gamma \in \Gamma$ satisfies
\[
\gamma^{(T_G^{X}(\abs{\gamma}_Y) \cdot (\abs{\gamma}_Y+1))!} = 1,
\]
where $\abs{\gamma}_Y$ denotes the word length of $\gamma$ with respect to $Y$.
In particular, $\Gamma$ is a torsion group and $T_{\Gamma}^{Y}$ is bounded above by the function $n \mapsto (T_G^{X}(n) \cdot (n+1))!$.
\end{corollary}

%

\subsection{Families of actions}

The crucial point of Corollary~\ref{cor:fundamental-torsion} is that the function
\[
n \mapsto (T_G^{X}(n) \cdot (n+1))!
\]
does neither depend on the action of $\Gamma$ on $\Omega^{+}$ nor on the choice of the point $p \in \Omega$.
This allows us to apply Corollary~\ref{cor:fundamental-torsion} simultaneously on a family of $G$-actions.
To do so, we consider a family $(\Omega_i)_{i \in I}$ of $G$-sets $\Omega_i$.
Let $\alpha_i \colon G \rightarrow \Sym(\Omega_i)$ denote the homomorphism corresponding to the action of $G$ on $\Omega_i$.
For each $i \in I$ we fix an element $p_i \in \Omega_i$ and let $(q_i)_{i \in \Omega}$ be a family of pairwise different elements that do not lie in $\cup_{i \in \N} \Omega_i$.
Let $\Omega_i^{+} \defeq \Omega_i \cup \{q_i\}$ and let $\tau_i = (p_i,q_i) \in \Sym(\Omega_i^{+})$.
We consider the homomorphism
\[
\alpha_I \colon G \rightarrow \prod \limits_{i \in I} \Sym(\Omega_i^{+}),\ g \rightarrow (\alpha_i(g))_{i \in I}
\]
and the sequence $\tau_I \defeq (\tau_i)_{i \in I} \in \prod_{i \in I} \Sym(\Omega_i^{+})$.
Let $\Gamma_I$ denote the subgroup of $\prod \limits_{i \in I} \Sym(\Omega_i^{+})$ that is generated by $\alpha_I(G)$ and $\tau_I$ and let $X_{I} \defeq \alpha_I(X) \cup \{\tau_I\}$, which is a finite generating set of $\Gamma_I$.

\begin{proposition}\label{prop:uniform-bound-orbit}
The torsion function $T_{\Gamma_I}^{X_I}$ of $\Gamma_I$ with respect to $X_I$ satisfies
\[
T_{\Gamma_I}^{X_I}(n) \leq (T_G^{X}(n) \cdot (n+1))!
\]
for every $n \in \N$.
In particular, $\Gamma_I$ is a torsion group.
\end{proposition}
\begin{proof}
The claim directly follows by applying Corollary~\ref{cor:fundamental-torsion} simultaneously to the actions of $\Gamma_I$ on $\Omega_i^{+}$, which are given by the canonical homomorphisms $\Gamma_I \rightarrow \Sym(\Omega_i^{+})$ for every $i \in I$.
\end{proof}

\subsection{The residually finite case}

Let us now assume that $G$ is an infinite finitely generated residually finite torsion group.
In this case we can choose a properly decreasing chain $(N_i)_{i \in \N}$ of finite index normal subgroups of $G$ that satisfies $\cap_{i \in \N} N_i = 1$.
Let $\Omega_i \defeq G / N_i$ and let $\alpha_i \colon G \rightarrow \Sym(\Omega_i)$ denote the action of $G$ that is given by left translation.
Then, using the assumption that $(N_i)_{i \in \N}$ is properly decreasing, we see that the homomorphism
\[
\alpha_{\geq n} \colon G \rightarrow \prod \limits_{i \geq n}^{\infty} \Sym(\Omega_i),\ g \rightarrow (\alpha_i(g))_{i \geq n}
\]
is injective for every $n \in \N$.
As before, we fix an element $p_i \in \Omega_i$ for each $i \in \N$ and a family $(q_i)_{i \in \N}$ of pairwise different elements that do not lie in $\cup_{i \in \N} \Omega_i$.
We write $\Omega_i^{+} \defeq \Omega_i \cup \{q_i\}$ and consider the elements $\tau_i = (p_i,q_i) \in \Sym(\Omega_i^{+})$ and $\tau \defeq (\tau_i)_{i \in \N} \in \prod_{i=1}^{\infty} \Sym(\Omega_i^{+})$.
Let $\Gamma \leq \prod \limits_{i \geq n}^{\infty} \Sym(\Omega_i^{+})$ denote the subgroup that is generated by $\alpha_{\geq 1}(G)$ and $\tau$.

\begin{lemma}\label{lem:extending-to-symmetric}
Let $Y$ be a finite set and let $H \leq \Sym(Y)$ be a subgroup that acts transitively on $Y$.
Let $Y^{+} = Y \cup \{z\}$, where $z \notin Y$.
For every $y \in Y$ the group $\Sym(Y^{+})$ is generated by $G$ and the transposition $(y,z)$.
\end{lemma}
\begin{proof}
Let $H$ denote the subgroup of $\Sym(Y^{+})$ that is generated by $G$ and $(y,z)$.
Since $G$ acts transitively on $Y$ it follows that every transposition of the form $(x,z)$ with $x \in Y$ is a conjugate of $(y,z)$ in $H$ and therefore lies in $H$.
By conjugating such a transposition $(x,z)$ with a transposition $(x',z)$, where $x' \notin \{x,z\}$, we obtain $(x,z)^{(x',z)} = (x,x') \in H$.
Now the proof follows from the well-known fact that $\Sym(Y^{+})$ is generated by all transpositions in $\Sym(Y^{+})$.
\end{proof}


Recall that a subgroup $H$ of a product of groups $P = \prod_{i \in I} K_i$ is called \emph{subdirect} if the canonical map $H \rightarrow K_i$ is surjective for every $i \in I$.

\begin{corollary}\label{cor:extending-to-symmetric}
The subgroup $\Gamma$ of $\prod_{i=1}^{\infty} \Sym(\Omega_i^{+})$ is subdirect.
\end{corollary}
\begin{proof}
This is a direct consequence of Lemma~\ref{lem:extending-to-symmetric} and the definition of $\Gamma$.
\end{proof}



\begin{lemma}\label{lem:finite-index-subgroup}
Let $(n_i)_{i \in \N}$ be a sequence of pairwise different natural numbers, let $P \defeq \prod_{i=1}^{\infty} \Sym(n_i)$, and let $H \leq P$ be a finitely generated subdirect subgroup.
Let $\iota \colon H \rightarrow P$ denote the inclusion map.
For each $k \in \N$ let $\pr_{\geq k} \colon P \rightarrow \prod_{i=k}^{\infty} \Sym(n_i)$ denote the canonical projection.
There is a natural number $m$ such that the image of the group $K \defeq H \cap \prod_{i=1}^{\infty} \Alt(n_i)$ under $\pr_{\geq m} \circ \iota$ is a subdirect subgroup of $\prod_{i=m}^{\infty} \Alt(n_i)$.
\end{lemma}
\begin{proof}
Let $\pi \colon \prod_{i=1}^{\infty} \Sym(n_i) \rightarrow \prod_{i=1}^{\infty} \Sym(n_i)^{\ab} \cong \prod_{i=1}^{\infty} \F_2$ denote the abelianization.
Note that $K$ is the kernel of $\pi \circ \iota$.
Since $H$ is finitely generated, its image in $\prod_{i=1}^{\infty} \F_2$ is finite and thus $K$ has finite index, say $k$, in $H$.
Since there are only finitely many alternating groups that admit proper subgroups of index at most $k$, it follows that the canonical map $K \rightarrow \Alt(n_i)$ is surjective for almost every $i$.
Thus the lemma follows if $m$ is chosen big enough.
\end{proof}

We are now ready to prove the main result.

\begin{theorem}\label{thm:main}
Every finitely generated residually finite torsion group $G$ embeds into a torsion group that is residually in the class $\mathcal{FS}$.
\end{theorem}
\begin{proof}
It is shown in~\cite[Theorem 1.1]{KionkeScheslerRealizingSubgroups} that every finitely generated, residually finite torsion group embeds into a finitely generated, residually finite perfect torsion group.
Regarding this, we can assume that $G$ is perfect.
Let $(N_i)_{i \in \N}$ be a strictly decreasing sequence of finite index normal subgroups of $G$ with $\cap_{i \in \N} N_i = 1$ and let $\Omega_i = G / N_i$.
From Corollary~\ref{cor:extending-to-symmetric} we know that $G$ embeds in a finitely generated subdirect subgroup $\Gamma$ of $\prod_{i=1}^{\infty} \Sym(n_i)$, where $n_i = \abs{\Omega_i} + 1$.
In this case Lemma~\ref{lem:finite-index-subgroup} provides us with a number $m \in \N$ such that the projection image of $K \defeq \Gamma \cap \prod_{i=1}^{\infty} \Alt(n_i)$ in $\prod_{i=m}^{\infty} \Sym(n_i)$ is a subdirect subgroup of $\prod_{i=m}^{\infty} \Alt(n_i)$.
Since $G$ is perfect we have $G \leq K$.
Moreover the restriction of the projection $\prod_{i=1}^{\infty} \Sym(n_i) \rightarrow \prod_{i=m}^{\infty} \Sym(n_i)$ to $G$ is injective since the sequence $(N_i)_{i \in \N}$ was chosen to be decreasing.
Thus $G$ embeds into the image of $K$ in $\prod_{i=m}^{\infty} \Alt(n_i)$, which is a finitely generated subdirect torsion subgroup, which completes the proof.
\end{proof}

\bibliographystyle{amsplain}
\bibliography{literatur}

\end{document}